\documentclass[12 pt]{amsart}
\usepackage[utf8]{inputenc}
\usepackage{yfonts}
\usepackage{latexsym}
\usepackage{amsthm}
\usepackage{amsmath}
\usepackage{amssymb}
\usepackage{verbatim}
\usepackage{color}
\newtheorem{theorem}{Theorem}
\newtheorem{corollary}{Corollary} 
\newtheorem{lemma}[theorem]{Lemma}
\newtheorem{conj}[theorem]{Conjecture} 
\newtheorem{example}[theorem]{Example}
\newtheorem{remark}[theorem]{Remark}
\newtheorem{defi}[theorem]{Definition}
\def \n{\noindent }
\def \bs{\bigskip}
\def \F{\mathcal F}
\def \B{\mathcal B}

\def \Z{\mathbb Z}
\def \N{\mathbb N}
\def \Z{\mathbb Z}
\def \zs{z.s.s }
\def \zsb{z.s.s{\tiny b} }
\def \m{\mathsf{s}'}
\def \ls{\mathsf{s}}

\title{Avoiding Zero-Sum Subsequences of Prescribed Length over the integers
} 
\author[C. Augspurger et al.]{C. Augspurger, M. Minter, K. Shoukry, P. Sissokho, and K. Voss\\ \\
	Mathematics Department, Illinois State University \\
	Normal, IL 61790--4520, U.S.A.}
\thanks{{\em Corresponding author:} psissok@ilstu.edu}
\begin{document}
\begin{abstract}
Let $t$ and $k$ be a positive integers, and let $I_k=\{i\in \mathbb{Z}:\; -k\leq i\leq k\}$. 
Let $\mathsf{s}'_t(I_k)$ be the smallest positive integer $\ell$ such that every zero-sum sequence $S$ over $I_k$ of length $|S|\ge \ell$ contains a zero-sum subsequence of length $t$. If no such $\ell$ exists, then let $\m_t(I_k)=\infty$.

In this paper, we prove that $\mathsf{s}'_t(I_k)$ is finite if and only if every integer in $[1,D(I_k)]$ divides $t$, where $D(I_k)=\max\{2,2k-1\}$ is the Davenport constant of $I_k$. Moreover, we prove that if  $\mathsf{s}'_t(I_k)$ is finite, then $t+k(k-1)\leq \mathsf{s}'_t(I_k)\leq t+(2k-2)(2k-3)$. We also show that $\mathsf{s}'_t(I_k)=t+k(k-1)$ holds for $k\leq 3$ and 
conjecture that this equality holds for any $k\geq1$.
\end{abstract}
\keywords{zero-sum sequence over ${\mathbb Z}$; no zero-sum subsequence of a given length.}
\maketitle
\section{Introduction and Main results}\label{sec:intro}
We shall follow the notation in~\cite{Gr}, by Grynkiewicz.
Let $\N$ be the set of positive integers. Let $G_0$ a subset of an abelian group $G$.
A sequence over $G_0$ is an unordered list of terms in $G_0$,
where repetition is allowed. The 
set of all sequences over $G_0$ is denoted by 
$\F(G_0)$
A sequence with no term is called {\em trivial} or {\em empty}. 
If $S$ is a sequence with terms $s_i$, $1\leq i\leq n$, we 
write $S=s_1\cdot\ldots\cdot s_n=\prod_{i=1}^n s_i$.
We say that $R$ is a {\em subsequence} of $S$ if any term in 
$R$ is in $S$. If $R$ and $T$ are subsequences
of $S$ such that $S=R\cdot T$, then $R$ is the {\em complementary} sequence of $T$ in $S$, and vice versa. We also write $T=S\cdot R^{-1}$ and $R=S\cdot T^{-1}$. For every sequence $S=s_1\cdot\ldots\cdot s_n$ over $G_0$, 
\begin{itemize}
\item $-S=(-s_1)\cdot\ldots\cdot (-s_n)$ 
\item the {\em length} of $S$ is $|S|=n$;
\item the {\em sum} of $S$ is $\sigma(S)=s_1+s_2+\ldots+s_n$;
\item the {\em subsequence-sum} of $S$ is  $\Sigma(S)=\{\sigma(R):\mbox{ $R$ is a subsequence of $S$}\}$.
\end{itemize}
For any sequence $R$ over $G_0$ and any integer $d\geq 0$, 
\[\mbox{$R^{[0]}$ is the trivial sequence, and $R^{[d]}=\underbrace{R\cdot\ldots\cdot R}_d$ for $d>0$.}\]

A sequence with sum $0$ is called {\em zero-sum}. The 
set of all zero-sum sequences over $G_0$ is denoted by 
$\B(G_0)$. A zero-sum sequence is called {\em minimal} if it does not contain a proper zero-sum subsequence.
The {\em Davenport constant} of $G_0$, denoted by $D(G_0)$ is the maximum length of a minimal zero-sum sequence over $G_0$.
The research on zero-sum theory is quite extensive 
when $G$ is a finite abelian groups (e.g., see~\cite{C,EGZ,Ga,GG} and the references therein). However, there is less activity 
when $G$ is infinite (e.g., see~\cite{BCRSS,CSS} and the references therein). The study of the particular case 
$G=\Z^r$ was explicitly suggested by Baeth and Geroldinger~\cite{BG} due to their relevance to direct-sum decompositions of modules.  In a recent paper, 
 Baeth et al.~\cite{BGGS} studied the Davenport constant of 
 $G_0\subseteq \Z^r$. The Davenport constant of an interval in 
$\Z$ was first derived (see Theorem~\ref{thm:La}) by Lambert~\cite{La} (also see~\cite{DGS,SST,Si} for related work.)
In a recent paper, Plagne and Tringali~\cite{PT}, studied the Davenport constant of the cartesian product of intervals of $\Z$. 

For any integers $x$ and $y$ with $x\leq y$, 
let $[x,y]=\{i\in \Z:\; x\leq i\leq y\}$. 
For $k\in\N$, let $I_k=[-k,k]$. 
\begin{theorem}[Lambert~\cite{La}]\label{thm:La}
$D(I_k)=\max\{2,2k-1\}$ for any $k\in\N$.
\end{theorem}
For $G$ finite and $G_0\subseteq G$, let $\ls_t(G_0)$ be the smallest integer $\ell$ such that any sequence $S\in\F(G_0)$ of length $|S|\ge \ell$ contains a zero-sum subsequence of length $t$.
If $t={\rm exp}(G)$, then $\ls_t(G_0)$ is called the 
Erd\"os--Ginzburg--Ziv constant and is denoted by $\ls(G)$.
In 1961, Erd\"os--Ginzburg--Ziv~\cite{EGZ} proved that $\ls(\Z_n)=2n-1$.  Reiher~\cite{Re} proved that $\ls(\Z_p\oplus \Z_p)=4p-3$ for any prime $p$. In general, if $G$ has rank two, say $G = Z_{n_1} \oplus Z_{n_2}$ with $1 \le n_1 \mid n_2$, then $\mathsf s (G) = 2n_1+2n_2-3$  (see Theorem 5.8.3 in Geroldinger--Halter--Koch~\cite{GHK}). For groups of higher rank, we refer the reader to Fan--Gao--Zhong~\cite{FGZ}. More recently, 
Gao et al~\cite{GHPS} proved that for any integer $k\geq2$ and any finite $G$ with exponent $n={\rm exp}(G)$, if the 
difference $n-|G|/n$ is large enough, then 
$\ls_{kn}(G)=kn+D(G)-1$. 
 
Observe that if $G$ is torsion-free and $G_0\subseteq G$,
then for any nonzero $g\in G_0$ and for any $d\in\N$, the sequence $g^{[d]}\in\F(G_0)$ does not contain a zero-sum subsequence. Thus, we will work with the following analogue 
of $\ls_t(G_0)$.
\begin{defi}\footnote{This formulation was suggested to us by Geroldinger and Zhong~\cite{GZ}.}
For any subset $G_0\subseteq G$, let $\m_t(G_0)$ be the smallest positive integer $\ell$ such that any sequence $S\in\B(G_0)$ of length $|S|\ge \ell$ contains a zero-sum subsequence of length $t$.
If no such $\ell$ exists, then let $\m_t(G_0)=\infty$.
\end{defi}
If $t={\rm exp}(G)$ is finite, then we denote $\ls_t(G_0)$ by $\ls(G)$. Let $r\in\N$ and assume that  $G\cong \Z_n^r$. Then $G$ has {\rm Property D} if any sequence $S\in\F(G)$ of length $\ls(G)-1$ that does not 
admit a zero-sum subsequence  has the form $S=T^{[n-1]}$ 
for some $T\in\F(G)$. Zhong found the following interesting connections between  $\ls(G)$ and $\m(G)$ (see the Appendix 
for their proofs). 
\begin{lemma}[Zhong~\cite{Z}]\label{lem:QZ} Let $G$ be a finite abelian group.
	
	\n $(i)$  If $\gcd(\ls(G)-1,{\rm exp}(G))=1$,
	then $\m(G)=\ls(G)$. 
	
	\n $(ii)$ If $G\cong \Z_n^r$, with $n\geq 3$ and $r\geq2$.
	Suppose that $\ls(G)=c(n-1)+1$ and $G$ has Property D.
	If $\gcd(\ls(G)-1,n)=c$, then $\m(G)<\ls(G)$.
\end{lemma}
\begin{remark}[Zhong~\cite{Z}]\label{rmk:QZ}\
	
	\n $(i)$ If $G\cong \Z^2_n$ with $n$ odd, then $\m(G)=\ls(G)$.
	
	\n $(ii)$ If $G\cong \Z^2_{2^h}$ with $h\geq 2$, then $\m(G)=\ls(G)-1$.
\end{remark}

 In this paper, we prove the following results about $\m_t(I_k)$, where $I_k=[-k,k]$.
\begin{theorem}\label{thm:1} Let $k$ and $t$ be  positive integers.

\n $(i)$ $\m_t(I_k)$ is finite, then every integer in $[1,D(I_k)]$ divides $t$.

\n $(ii)$ If every integer in $[1,D(I_k)]$ divides $t$, then 
\[t+k(k-1)\leq  \m_t(I_k)\leq t+(2k-2)(2k-3).\]   
\end{theorem}
\begin{corollary}\label{cor:1} Let $t\in\N$ and $k\in[1,3]$. Then $\m_t(I_k)= t+k(k-1)$ if and only if  every integer in $[1,D(I_k)]$ divides $t$. 
\end{corollary}
\begin{conj}\label{conj:1} Corollary~\ref{cor:1} 
holds for any $k\in\N$.   
\end{conj}
\section{Proofs of the main results}\label{sec:main}
For any integers $a$ and $b$, we denote $\gcd(a,b)$ by $(a,b)$. We use the abbreviations \zs and \zsb 
for {\em zero-sum sequence(s)} and {\em zero-sum subsequence(s)}, respectively. The letters $k$ and $t$ will denote positive integers throughout the paper. 

\bs\n The following lemma gives a lower bound for $\m_t(I_k)$.
\begin{lemma}\label{lem:3.0}
Consider the \zs $U=k\cdot(-1)^{[k]}$ and $V=(k-1)\cdot(-1)^{[k-1]}$.
Then, $S=U^{[\frac{t}{k+1}-1]}\cdot V^{[k]}$ and $R=U^{[k-1]}\cdot V^{[\frac{t}{k}-1]}$ are \zs that do not contain a \zsb of 
length $t$. Thus, $\m_t(I_k)\geq t+k(k-1)$.
\end{lemma}
\begin{proof}
We prove the lemma for $S$ only since the proof 
for  $R$ is similar. By contradiction, assume that  $S$ contains a \zsb of length $t$. Since
$\sigma(S)=0$, it follows that $S$ also contains
a \zsb $S'$ of length $|S|-t=k(k-1)-1$.
Moreover, $S'$ can be written as $S'=k^{[a]}\cdot (k-1)^{[b]}\cdot (-1)^{[c]}$ for some nonnegative integers $a$, $b$, and $c$. Hence 
$\sigma(S')=ak+b(k-1)-c=0$ and $a+b+c=|S'|=k^2-k-1$. Thus  
\[(a+1)(k+1)=k(k-b).\]
Since $a,b,k\geq 0$, we have $0<k-b\leq k$. Since $(k,k+1)=1$, we obtain that $k+1$ divides $k-b$, which is a contradiction.
Thus $\m_t(I_k)\geq |S|+1=t+k(k-1)$. 
\end{proof}
\begin{example} For $k=3$, $S =\left(3\cdot-1\cdot-1\cdot-1\right)^{[14]} \cdot \left(2\cdot-1\cdot-1\right)^{[3]}$ is a \zs of length 
$65$ over $[-3,3]$ which does not contain a \zsb  of length $t=60$. 
\end{example}
\begin{lemma}\label{lem:3.1} Let $a,b,x\in\N$.
If $S=a^{[\frac{b}{(a,b)}]}\cdot (-b)^{[\frac{a}{(a,b)}]}$ is a z.s.s, then the length of any \zsb of $S^{[x]}$ is a multiple of 
$\vert S\vert$.
\end{lemma}
\begin{proof}  
Let $S'$ be a \zsb of $S^{[x]}$. Since the terms of $S$ are $a$ 
and $-b$, there exist nonnegative integers $h$ and $r$ such that 
 $S'=a^{h}\cdot (-b)^{r}$ and
\begin{equation}\label{1}
\sigma(S')=ha-rb=0\Rightarrow h\dfrac{a}{(a,b)}=r\dfrac{b}{(a,b)}.
\end{equation}
Since $\left(\frac{b}{(a,b)},\frac{b}{(a,b)}\right)=1$, we obtain $\frac{b}{(a,b)}$ divides $h$ and $\frac{a}{(a,b)}$
divides $r$.  Thus, $h=p\frac{b}{(a,b)}$ and $r=q\frac{a}{(a,b)}$ for some integers $p$ and $q$.  Substituting $h$ and $r$ back into~\eqref{1} yields $p=q$. Thus,
\[|S'|=h+r=p\frac{b}{(a,b)}+q\frac{a}{(a,b)}=p|S|.\]
\end{proof}
\begin{lemma}\label{lem:3.2}  If $\m_t(I_k)$ is finite, then every odd integer in $[1,D(I_k)]$ divides $t$.
\end{lemma}
\begin{proof}
Since the lemma is trivial for $k=1$, we assume that 
$k\geq 2$. Then
$D(I_k)=2k-1$ by Theorem~\ref{thm:La}. 
Let $\ell=2c-1$ be an odd integer in $[3,D(I_k)]$, and 
consider the minimal \zs $S=c^{[c-1]}\cdot (-c+1)^{[c]}$. 
Then, for any $x\in\N$, it follows
from Lemma~\ref{lem:3.1} that for any \zsb $R$ of $S^{[x]}$, 
$|R|$ divides $|S|=2c-1=\ell$. Since 
$\ell$ does not divide $t$, there is no \zsb of 
$S^{[x]}$ whose length is equal to $t$. Since $x$ is arbitrary, it follows that $\m_t(I_k)$ can be arbitrarily large.
This proves the lemma.  
\end{proof}

\bs To prove the upper bound in Theorem~\ref{thm:1}$(ii)$, we will use the following lemma which is a directly application 
a well-known fact: ``Any sequence of $n$ integers contains a nonempty subsequence whose sum is divisible by $n$''.
\begin{lemma}\label{lem:3.5}
Let $\beta\in\N$ and $X\in\F(\Z)$. If $|X|\geq \beta$, then there exists a factorization 
$X=X_0\cdot X_1\cdot\ldots\cdot X_r$ such that 
\begin{enumerate}
\item[(i)]  $|X_0|\leq \beta-1$ and no subsequence of $X_0$ 
has a sum that is divisible by $\beta$.
\item[(ii)] $|X_j|\leq \beta$ and  $\sigma(X_j)$ is divisible by $\beta$ for any $j\in[1,r]$.
\end{enumerate}
\end{lemma}
We will also use the following lemmas.
\begin{lemma}\label{lem:3.6}
Assume that $k\geq2$ and that every integer in $[1,D(I_k)]$ divides $t$. Let $S$ be a \zs over $I_k=[-k,k]$ that does not contain a \zsb of length $t$.  Let $S=S_1\cdot\ldots\cdot S_h$  
be a factorization of $S$ into minimal \zsb $S_i$, $1\leq i\leq h$. If $|S|\geq  t+k(k-1)$, then there exists some length $\beta$ such that 
$n_\beta=|\{S_i:\; |S_i|=\beta,\; 1\leq i\leq h\}|$ satisfies:
\[n_\beta> (2k-2)(2k-3).\]
\end{lemma}
\begin{proof}
Recall that $(a,b)$ denotes $\gcd(a,b)$. It is easy 
to see that 
\begin{equation}\label{lem3.6:eq1}
(2k-3,2k-2)=(2k-2,2k-1)=(2k-3,2k-1)=1.
\end{equation}
Since $k>1$ and every integer in $[1,D(I_k)]=[1,2k-1]$ is a factor of $t$, it follows from~\eqref{lem3.6:eq1}  
that $t=p(2k-1)(2k-2)(2k-3)$, for some $p\in\N$.  
By definition, we have $\max_{1\leq i\leq h}|S_i|\leq D(I_k)=2k-1$. Thus, it follows from the pigeonhole principle that there exists some length $\beta$ such that
\[n_\beta\geq\frac{t+k(k-1)}{\max\limits_{1\leq i\leq h}|S_i|}\geq \frac{t+k(k-1)}{2k-1}>p(2k-2)(2k-3).\]
\end{proof}
\begin{lemma}\label{lem:3.7}
Assume that $k\geq2$ and that every integer in $[1,D(I_k)]$ divides $t$. Let $S$ be a \zs over $I_k=[-k,k]$ of 
length $|S|\geq t+k(k-1)$ such that $S$ does not contain a \zsb of length $t$. Let $S=S_1\cdot\ldots\cdot S_h$  
be a factorization of $S$ into minimal \zsb $S_i$, $1\leq i\leq h$. Let $L=\{|S_i|:\; 1\leq i\leq h\}$, $\alpha=\max_{\ell\in L}\ell$, and let $n_\ell=|\{S_i:\; |S_i|=\ell,\; 1\leq i\leq h\}|$.

If there exists $\beta\in L$ such that $n_\beta\geq \alpha-1$, then 
\[|S|\leq t-\beta+(\beta-1)\max_{\ell\in L\setminus\{\beta\}}\ell.\]
\end{lemma}
\begin{remark}\label{rmk:1}
By Lemma~\ref{lem:3.6}, there exists $\beta\in L$ such that $n_\beta>(2k-2)(2k-3)$. Moreover, $\alpha=\max_{\ell\in L}\ell\leq D(I_k)\leq(2k-2)(2k-3)+1$ for $k\geq2$. Thus, $n_\beta\geq \alpha$, i.e., the hypothesis of Lemma~\ref{lem:3.7} always holds.
\end{remark}
\begin{proof}
By hypothesis, there exists $\beta\in L$ such that $n_\beta\geq \alpha-1$.
Given a factorization $S=S_1\cdot\ldots\cdot S_h$ 
into minimal \zs $S_i$, $1\leq i\leq h$, consider the sequence
of lengths in $L\setminus\{\beta\}$:
\[X=\prod_{i=1,\,|S_i|\not=\beta}^h|S_i|=\prod_{\ell\in L\setminus\{\beta\}}\ell^{[n_\ell]}.\]
 It follows from Lemma~\ref{lem:3.5} that there exists a factorization $X=X_0\cdot X_1\ldots X_r$ such that 
\begin{equation}\label{eq:a2.1}
\mbox{$|X_0|\leq \beta-1$, and no subsequence $X_0$ 
	has a sum that is divisible by $\beta$.}
\end{equation}
\begin{equation}\label{eq:a2.2}
\mbox{ $|X_j|\leq \beta$ and $\beta$
	 divides $\sigma(X_j)$ for all $j\in[1,r]$}.
\end{equation}	
Thus, 
\begin{equation}\label{eq:a2.3}
\mbox{ $\sigma(X_j)=\sum\limits_{x\in X_j}x\leq |X_j|\cdot\max\limits_{x\in X_j}x\leq \beta\alpha$ for all $j\in[1,r]$}.
\end{equation}
Note that \eqref{eq:a2.2}, \eqref{eq:a2.3}, and the hypothesis on
 $\beta$ imply that: 
 \[\mbox{$\beta$  divides $t$; $n_\beta\geq\alpha-1$; $\sigma(X_j)\leq\alpha\beta$; and $\beta$  divides $\sigma(X_j)$ for all $j\in[1,r]$.}\]
 Thus, if 
 \[\beta n_\beta+\sum_{j=1}^r\sigma(X_j) \geq t,\]
 then there exists a nonnegative integer 
 $n_\beta'\leq n_\beta$ and a subset $Q\subseteq[1,r]$ such that 
 \[\beta n_\beta'+\sum_{q\in Q}\sigma(X_q) =t.\]
Then $S$ would contain a \zs of length $t$ obtained 
by concatenating $n_\beta'$ minimal \zsb of $S$ 
of length $\beta$ and all the \zs of $S$ whose lengths are in $X_q$ for all $q\in Q$. This contradicts the hypothesis of the theorem. Thus, $\beta n_\beta+\sum_{j=1}^r\sigma(X_j)<t$ must hold. Since $\beta$ divides both $t$ and $\sum_{j=1}^r\sigma(X_j)$, we obtain 
\[\beta n_\beta+\sum_{j=1}^r\sigma(X_j) \leq t-\beta.\] 
Thus, it follows from the definition of $X$ and $X_j$, $0\leq j\leq r$, that 
\begin{align}\label{eq:a2.6}
|S|=\sum_{\ell\in L} \ell n_\ell
&=\beta n_\beta+\sigma(X)\cr
&=\beta n_\beta+\sum_{j=1}^r\sigma(X_j)+\sigma(X_0)\cr
&\leq t-\beta+\sigma(X_0).
\end{align}
Next, it follows from \eqref{eq:a2.1} and \eqref{eq:a2.6} that 
\begin{equation*}\label{eq:a2.7}
|S|\leq t-\beta+\sigma(X_0)\leq t-\beta+|X_0|\max_{\ell\in  L\setminus\{\beta\}}\ell\leq t-\beta+(\beta-1)\max_{\ell\in L\setminus\{\beta\}}\ell.
\end{equation*}
\end{proof}
\begin{proof}[Proof of Theorem~\ref{thm:1}]
We first prove part $(i)$. Suppose that $\m_t(I_k)$ is finite.
Then it follows from Lemma~\ref{lem:3.2} that every odd integer in $[1,D(I_k)]$ divides $t$. Thus, it remains to show that 
if $a$ is an even integer in $[1,D(I_k)]$, then $a$ divides $t$.

\n\textbf{Case 1:} $a=2^{e}$ for some integer $e\geq1$.

	Lemma~\ref{lem:3.1} implies that for any $p\in\N$, the sequence $S=\left(1\cdot-1\right)^{[p]}$ is a \zs whose 
\zsb have lengths that are multiples of $2$.  Therefore, if $2$ does not divide $t$, then $\m_t(I_k)
\geq |S|=2p$, where $p$ can be chosen to be arbitrarily large. Thus, $2$ divides $t$ if $\m_t(I_k)$ is finite.

Now assume that $e>1$. Since the $\gcd$ of two numbers divides their difference, $(a/2-1,a/2+1)\leq2$.  
But $2$ does not divide $a/2-1$ or $a/2+1$; and so 
$(a/2-1,a/2+1)=1$. Lemma~\ref{lem:3.2} implies that for any $p\in\N$, the sequence $S^{[p]}$ with 
$S= (a/2-1)^{[a/2+1]} \cdot (-a/2-1)^{[a/2-1]}$ is 
a \zs whose \zsb have lengths that are 
multiples of $|S|=(a/2+1)+(a/2-1)=a$.  Thus, 
if $a$ does not divide $t$, we can obtain arbitrarily 
long \zs over $I_k=[-k,k]$ that do not contain \zsb of length $t$,
because $p$ can be chosen to be arbitrarily large. 
Thus, $a$ divides $t$ if $\m_t(I_k)$ is finite.

\n\textbf{Case 2:}  $a$ is not a power of $2$. 

Then $a=2^{e}j$, where $e$ and $j$ are nonnegative 
integers such that $j$ is odd.  By Lemma~\ref{lem:3.2}, $j$ divides $t$, and 
if follows from Case~$1$ that $2^{e}$ divides $t$.  
Since $j$ is odd, $(2^e,j)=1$. Since $2^e$ and $j$ are factors of $t$, it follows that $2^ej$ divides $t$.

Thus, it follows from Case~$1$, Case~$2$, and Lemma~\ref{lem:3.2} that every  integer in $[1,D(I_k)]$ divides $t$.

\bs Since the lower bound of $\m_t(I_t)$ in Theorem~\ref{thm:1}$(ii)$ follows from Lemma~\ref{lem:3.0}, it remains to prove its upper bound.
Let $k,t\in\N$ be such that every integer in $[1,D(I_k)]$ divides $t$. In particular, $t$ is even.
Let $S$ be an arbitrary \zs over $I_k=[-k,k]$ that does not contain a \zsb of length $t$. 

If $k=1$, then it follows from Theorem~\ref{thm:La} that $D(I_k)=2$. Thus, $2$ divides $t$ and $|S|=x_1+2x_2$ for some nonnegative integers $x_1$ and $x_2$.  
If $|S|\geq t$, then $x_1\geq2$ or $x_2\geq t/2$ (because $t$ is even). This implies 
that there exit nonnegative integers $x_1'\leq x_1$ and $x_2'\leq x_2$ such that $x_1'+2x_2'=t$. Thus $S'=(1\cdot -1)^{[x_2']}\cdot 0^{[x_1']}$ is a \zsb of $S$ of length $t$, which contradicts 
the fact that $S$ does not contain a \zs of length $t$. Hence $|S|\leq t-1$, and $\m_t(I_k)\leq |S|+1=t$.

Now assume $k\geq2$.  Since $S$ 
was arbitrarily chosen,  it follows that if $|S|\leq t+k(k-1)-1$, then 
\[\m_t(I_k)\leq |S|+1\leq t+k(k-1)\leq t+(2k-2)(2k-3),\]
and the upper bound in Theorem~\ref{thm:1}$(ii)$ follows.
So we may assume that $|S|\geq t+k(k-1)$. Let $S=S_1\cdot\ldots\cdot S_h$  
be a factorization of $S$ into minimal \zsb. Let $L=\{|S_i|:\; 1\leq i\leq h\}$, $\alpha=\max_{\ell\in L}\ell$, and let $n_\ell=|\{S_i:\; |S_i|=\ell,\; 1\leq i\leq h\}$.
Then Remark~\ref{rmk:1} implies that there exists $\beta\in L$ is such that $n_\beta\geq \alpha-1$. 
If  $\beta=\alpha$, then Lemma~\ref{lem:3.7} yields
\begin{align*}
|S|\leq t-\alpha+(\alpha-1)\max_{\ell\in L\setminus\{\alpha\}}\ell\leq t-\alpha+(\alpha-1)^2.
\end{align*}
If  $1\leq \beta\leq \alpha-1$, then Lemma~\ref{lem:3.7} also yields 
\begin{align*}
|S|
&\leq t+\max_{1\leq \beta\leq \alpha-1}\left(-\beta+(\beta-1)\max_{\ell\in L\setminus\{\beta\}}\ell\right)
\cr
&\leq t+\max_{1\leq \beta\leq \alpha-1}\left(-\beta+(\beta-1)\alpha\right)\cr
&= t+\left(-(\alpha-1)+(\alpha-2)\alpha\right)\cr
&= t-\alpha+(\alpha-1)^2.\cr
\end{align*}
So in all cases, we obtain
\begin{align}\label{eq:uba}
|S|\leq t-\alpha+(\alpha-1)^2\leq t-(2k-1)+(2k-2)^2,
\end{align}
where we used the fact $\alpha\leq D(I_k)=2k-1$.
Since $S$ was chosen to be an arbitrary \zs over $I_k=[-k,k]$ which does not contain a \zsb of length $t$, it follows that 
\[\m_t(I_k)\leq |S|+1\leq t-(2k-1)+(2k-2)^2+1=t+(2k-2)(2k-3).\]
\end{proof}
\begin{proof}[Proof of Corollary~\ref{cor:1}]
For $k\in\{1,2\}$, the corollary holds since the upper and lower bounds of $\m_t(I_k)$ given by Theorem~\ref{thm:1} are both equal to $t+k(k-1)$. 

For $k=3$, it also follows from Theorem~\ref{thm:1} that 
$t+6\leq \m_t(I_3)\leq t+12$. Thus, it remains to show that if 
$S$ is an arbitrary \zs over $I_3$ which does not contain a \zs 
of length $t$, then $|S|\not=t+d$ for all $d\in[6,11]$. 

Consider a factorization $S=S_1\cdot\ldots\cdot S_h$ into minimal \zs $S_i$, $i\in[1,h]$. Let $L=\{|S_i|:\; 1\leq i\leq h\}$, $\alpha=\max_{\ell\in L}\ell$, and let $n_\ell=|\{S_i:\; |S_i|=\ell,\; 1\leq i\leq h\}|$. Thus,  $\alpha\leq D(I_3)=5$. If $\alpha\leq4$, then Lemma~\ref{lem:3.7} yields
\[|S|\leq  t+\max_{1\leq\alpha\leq 4}\left((\alpha-1)^2-\alpha\right)= t+(4-1)^2-4=t+5.\]
Thus, we may assume that $\alpha=\max\limits_{\ell\in L}\ell=5$ for any factorization of $S$.

If $\beta\in\{1,2\}$ and $n_\beta\geq 4$, then  Lemma~\ref{lem:3.7} yields
\[|S|\leq  t+\max_{\beta\in\{1,2\}}\left( (\beta-1)\alpha-\beta\right)= t+(2-1)5-2=t+3.\]
Next, suppose that $R$ is a \zsb of $S$ with length at least $4$.
Then $R\cdot -R$ can be trivially factorize into $|W|\geq 4$  \zs of length $2$. This would yields a new factorization $S=S'_1\cdot\ldots\cdot S'_h$ with $n_2\geq 4\geq n_5-1$,
which would imply that $|S|< t+5$ by the above analysis.

Also note that if $n_\ell\geq t/\ell$ holds for some length $\ell\in L$, then we obtain a 
\zsb of $S$ of length $t$ by concatenating $t/\ell$ \zsb of length $\ell$ in $S$. This would contradict the definition of $S$.  
Thus, we can assume that $n_\ell\leq t/\ell-1$ for all $\ell\in L\subseteq [1,5]$.

To recapitulate, we may assume that for any factorization $S=S_1\cdot\ldots\cdot S_h$, with $S_L=\prod_{i=1}^h |S_i|$ and $n_\ell=|\{S_i:\; |S_i|=\ell,\; 1\leq i\leq h\}|$, we have:  
\begin{enumerate}
\item[(i)] $S_L=5^{[n_5]}\cdot 4^{[n_4]}\cdot 3^{[n_3]}\cdot 2^{[n_2]}\cdot1^{[n_1]}$, where $0\leq n_\ell\leq t/\ell-1$ for $\ell\in[1,5]$; $n_5\geq 1$; and $n_1,n_2\leq 3$.
\item[(ii)] There is a one-to-one correspondence between 
the subsequences $S'_L$ of $S_L$ and the \zs $S'$ of $S$ with length 
$\sigma(S'_L)$. 
\item[(iii)]  If $R$ is \zs over $I_3$ such that $|R|\geq 4$, then $R$ and $-R$ cannot both be subsequences of $S$.
\item[(iv)] If $R$ is a minimal \zsb  of $S$ such that $|R|=5$,
then $R=3^{[2]}\cdot(-2)^{[3]}$. \\ 
(This follows from $(iii)$ and the fact $A=3^{[2]}\cdot(-2)^{[3]}$ and $-A$ are the only minimal \zs of length $5$ over $I_3=[-3,3]$. Thus, if $-A$ is the \zsb of $S$, then we can analyze $-S$ instead of $S$.)
\end{enumerate}
 
We now prove the following claims.

\bs\n {\bf Claim~1:} {\em  If $5\cdot 3^{[4]}$ is a subsequence of $S_L$, then  $|S|\not=t+d$ for all $d\in[6,11]$}

If $n_4+n_2+n_1\geq1$, then either $5\cdot 4\cdot 3^{[4]}$, or  $5\cdot 3^{[4]}\cdot 2$, or $5\cdot 3^{[4]}\cdot 1$ is a subsequence of $S_L$, which implies that $\Sigma(S_L)$ contains all the integers in $[6,11]$. Thus, ${n_4}={n_2}={n_1}=0$, which implies that $S_L=5^{[n_5]}\cdot 3^{[n_3]}$. If $n_5\leq 1$, then 
\[|S|=\sigma(S_L)=5n_5+3n_3\leq 5+3(t/3-1)<t+5.\]
Thus, we may assume that  
 
\centerline{$S_L=5^{[n_5]}\cdot 3^{[n_3]}$, where $n_5\geq 2$ and ${n_3}\geq 4$.}

Then $\Sigma(S_L)$ contain all the integers in $[6,11]\setminus\{7\}$; and so $|S|\not=t+d$ for $d\in[6,11]\setminus\{7\}$. It remains to show that $|S|\not=t+7$.

Note that the only minimal \zs of length $3$ over
$[-3,3]$ are (up to sign) $B_1=2\cdot(-1)^{[2]}$ and $B_2=3\cdot-2\cdot-1$. Since $5\cdot 3^{[4]}$ is a subsequence of $S_L$, it follows from the assumptions (i)--(iv) (see above) that 
$S'=A\cdot X\cdot Y\cdot Z\cdot W$ is a subsequence of $S$, where  $A=3^{[2]}\cdot(-2)^{[3]}$ and $X,Y,Z,W\in\{-B_1,B_1,-B_2,B_2\}$. By inspecting the sequence $S'$ for all possible choices of $X$, $Y$, $Z$, and $W$; we see that $S'$ admits a \zs of length $7$. For instance, if $X=Y=Z=B_2$, then 
\[S'=A\cdot B_2^{[3]}\cdot W=A^{[2]}\cdot 3\cdot(-1)^{[3]}\cdot W\]
contains the subsequence $3\cdot(-1)^{[3]}\cdot W$, which is a \zs of length $4+|W|=7$. Hence, $|S|\not=t+7$.
 Thus, $|S|\not=t+d$ for all $d\in[6,11]$.
 
\bs\n {\bf Claim~2:} {\em If $5\cdot 4^{[2]}\cdot 3$ is a subsequence of $S_L$, then  $|S|\not=t+d$ for all $d\in[6,11]$.}

If $n_3\geq2$, or $n_2\geq 1$, or $n_1\geq1$, then either $5\cdot 4^{[2]}\cdot 3^{[2]}$, or $5\cdot 4^{[2]}\cdot 3\cdot 2$, or  $5\cdot 4^{[2]}\cdot 3\cdot 1$ is a subsequence of $S_L$, which implies that $\Sigma(S_L)$ contains 
all the integers in $[6,11]$. In these cases, $|S|\not=t+d$ for $d\in[6,11]$, we are done. Thus, we may assume that $n_2=n_1=0$
and $n_3=1$, which implies that $S_L=5^{[n_5]}\cdot 4^{[n_3]}\cdot3$. 
If $n_5\leq1$, then 
\[|S|=\sigma(S_L)=5n_5+4n_4+3\leq 5+4(t/4-1)+3<t+5.\] 
Thus, we may assume that  

\centerline{$S_L=5^{[n_5]}\cdot 4^{[n_4]}\cdot 3$, where $n_5\geq 2$ and ${n_4}\geq 2$.}

Thus, $5^{[2]}\cdot 4^{[2]}\cdot 3$ is a subsequence of $S_L$, which implies that $\Sigma(S_L)$ contain all the integers in $[7,11]$. Thus $|S|\not=t+d$ for $d\in[7,11]$. It remains to 
show that $|S|\not=t+6$.

Note that the only minimal \zs of length $4$ over
$[-3,3]$ are (up to sign) $C_1=3\cdot(-1)^{[3]}$ and $C_2=3\cdot 1 \cdot(-2)^{[2]}$. Since $5\cdot 4^{[2]}\cdot 3$ is a subsequence of $S_L$, it follows from the assumptions (i)--(iv) that 
$S'=A\cdot X\cdot Y\cdot Z$ is a subsequence of $S$, where  $A=3^{[2]}\cdot(-2)^{[3]}$, $X,Y\in\{-C_1,C_1,-C_2,C_2\}$, 
and $Z\in\{-B_1,B_1,-B_2,B_2\}$. By inspecting the sequence $S'$ for all possible choices of $X$, $Y$, and $Z$; we see that $S'$ admits a \zs of length $6$. 
For instance, if $X=C_1$ and $Y=C_2$, then 
\[S'=A\cdot C_1\cdot C_2\cdot Z=A\cdot (3\cdot -1\cdot-2)^{[2]}\cdot (1\cdot-1)\cdot Z\]
contains the subsequence $(3\cdot -1\cdot-2)\cdot Z$, which is a \zsb of length $3+|Z|=6$. Hence, $|S|\not=t+6$.
Thus, $|S|\not=t+d$ for all $d\in[6,11]$.

\bs\n {\bf Claim~3:} {\em If $5\cdot 4^{[3]}$ is a subsequence of $S_L$, then  $|S|\not=t+d$ for all $d\in[6,11]$.}

If $n_3\geq 1$, then $5\cdot 4^{[2]}\cdot3$ is a subsequence of $S_L$, and we are back in Case~2. Thus, we may assume that $n_3=0$. If ${n_2}\geq1$ or ${n_1}\geq2$, then either  $5\cdot 4^{[3]}\cdot2$ or  $5\cdot 4^{[3]}\cdot1^{[2]}$ is a subsequence of $S_L$, which implies that $\Sigma(S_L)$ contains all the integers in $[6,11]$. Thus, $S$ contains \zsb of length $\ell$ for all $\ell\in[6,11]$. Hence, $|S|\not=t+d$ for $d\in[6,11]$. Thus, we may assume that $n_2=0$ and $n_1\leq 1$. Thus,  $S_L=5^{[n_5]}\cdot 4^{[n_3]}\cdot 1^{[n_1]}$. Moreover, if $n_5\leq1$, then 
\[|S|=\sigma(S_L)=5n_5+4n_4+{n_1}\leq 5+4(t/4-1)+1<t+5.\]
Thus, we may assume that 

\centerline{$S_L=5^{[n_5]}\cdot 4^{[n_4]}\cdot 1^{[n_1]}$, where $n_5\geq 2$, $n_4\geq 3$, and $n_1\leq 1$.}
Since $5^{[2]}\cdot 4^{[3]}$ is a subsequences of $S_L$, 
it follows that $\Sigma(S_L)$ contain all the integers in $[8,10]$.
Thus, $S$ admits \zs of length $\ell$ for all $\ell\in[8,10]$.
Hence, $|S|\not=t+d$ for all $d\in[8,10]$.
Moreover, it follows from the assumptions (i)--(iv) that 
$S'=A\cdot X\cdot Y\cdot Z$ is a subsequence of $S$, where  $A=3^{[2]}\cdot(-2)^{[3]}$ and $X,Y,Z\in\{-C_1,C_1,-C_2,C_2\}$.
By inspecting the sequence $S'$ for all possible choices 
of $X$, $Y$, and $Z$; we see that $S'$ admits a \zs of length $7$. Hence, $|S|\not=t+7$. Overall, we obtain $|S|\not=t+d$ for any $d\in[7,10]$.

If $5\cdot 4^{[4]}$ is a subsequence of $S_L$, it again follows from the assumptions (i)--(iv) that 
$S'=A\cdot X\cdot Y\cdot Z\cdot W$ is a subsequence of $S$, where  $A=3^{[2]}\cdot(-2)^{[3]}$ and $X,Y,Z,W\in\{-C_1,C_1,-C_2,C_2\}$.
By inspecting the sequence $S'$ for all possible choices 
of $X$, $Y$, $Z$, and $W$; we see that $S'$ admits \zs of lengths $6$ and $11$. In this case, $|S|\not=t+d$ for all $d\in[6,11]$. 
Thus, we may assume that 

\centerline{$S_L=5^{[n_5]}\cdot 4^{[3]}\cdot 1^{[n_1]}$, where $n_5\geq 2$ and $n_1\leq 1$.}

\n Now, it remains to show that $|S|\not=t+a$ for $a\in\{6,11\}$.
However, if $|S|=t+a$, then 
\[5{n_5}+4(3)+{n_1}=\sigma(S_L)=|S|=t+a\Rightarrow 5{n_5}=t+a-12-{n_1}.\] This  is a contradiction since $5$ divides $t$ (by hypothesis) and $5$ does not divide $a-12-{n_1}$
for $a\in\{6,11\}$ and ${n_1}\in\{0,1\}$. Thus, $|S|\not=t+d$ for all $d\in[6,11]$.

\bs Based on Claim~1--Claim~3, we may assume the following:
\begin{enumerate}
	\item[(v)] $S_L=5^{[n_5]}\cdot 4^{[n_4]}\cdot 3^{[n_3]}\cdot 2^{[n_2]}\cdot1^{[n_1]}$, where $0\leq n_\ell\leq t/\ell-1$ for all $\ell\in[1,5]$; ${n_1},{n_2},{n_3}\leq3$; ${n_4}\leq 2$; $(n_4,n_3)\not=(2,1)$;  and $n_5\geq1$.
\end{enumerate}
We will use this assumption in the following cases.

\bs\n {\bf Case 1:} {\em $|S|\not=t+6$.}

Assume that $|S|=t+6$. If ${n_1}\geq1$,
then $5\cdot 1$ is a subsequence of $S_L$, which 
implies that $S$ contains a \zsb of length $5+1=6$ 
whose complementary sequence in $S$ is a \zsb of length $t$.  Thus,
${n_1}=0$. By a similar reasoning, we infer that the following 
conditions hold: ${n_3}\leq 1$; and $n_4\geq 1\Rightarrow{n_2}=0$. 
Moreover, it follows from condition (v) that $(n_4,n_3)\not=(2,1)$.
Consequently, either $S_L=5^{[n_5]}\cdot 4^{[n_4]}\cdot 3^{[n_3]}$
with $n_3\leq 1$, ${n_4}\leq 2$, and $(n_4,n_3)\not=(2,1)$; or  $S_L=5^{[n_5]}\cdot 3^{[n_3]}\cdot2^{[n_2]}$
with $n_3\leq 1$ and ${n_2}\leq 2$. Thus,
\[|S|=\sigma(S_L)\leq 5{n_5}+8\leq 5(t/5-1)+8<t+6,\]
which is a contradiction. Thus, $|S|\not=t+6$.

\bs
\bs\n {\bf Case 2:} {\em $|S|\not=t+7$.}

Assume that $|S|=t+7$. If $n_2\geq1$,
then $5\cdot 2$ is a subsequence of $S_L$, which 
implies that $S$ contains a \zsb of length $5+2=7$ 
whose complementary sequence in $S$ is a \zsb of length $t$.  Thus,
$n_2=0$. By a similar reasoning, we infer that the following 
conditions hold:  ${n_1}\leq 1$;  
${n_4}\geq1\Rightarrow {n_3}=0$; ${n_3}\geq1\Rightarrow {n_4}=0$; and ${n_3}\geq2\Rightarrow {n_1}=0$. Consequently, either $S_L=5^{[n_5]}\cdot 4^{[n_4]}\cdot 1^{[n_1]}$
with ${n_4}\leq2$ and $n_1\leq 1$; or $S_L=5^{[n_5]}\cdot 3\cdot 1$, or $S_L=5^{[n_5]}\cdot 3^{[n_3]}$ with $n_3\leq3$. Thus, 
\[|S|=\sigma(S_L)\leq 5{n_5}+9\leq 5(t/5-1)+9<t+7,\]
which is a contradiction. Thus  $|S|\not=t+7$.

\bs\n {\bf Case 3:} {\em $|S|\not=t+8$.}

Assume that $|S|=t+8$. If ${n_3}\geq1$,
then $5\cdot 3$ is a subsequence of $S_L$, which 
implies that $S$ contains a \zsb of length $5+3=8$ 
whose complementary sequence in $S$ is a \zsb of length $t$. Thus,
${n_3}=0$. By a similar reasoning, we infer that ${n_4}\leq1$; 
${n_2}\leq 3$; ${n_1}\leq 2$; ${n_2}\geq1\Rightarrow {n_1}=0$; and ${n_1}\geq1\Rightarrow {n_2}=0$. Consequently either $S_L=5^{[n_5]}\cdot 4\cdot2$, or $S_L=5^{[n_5]}\cdot 4\cdot1^{[n_1]}$, or $S_L=5^{[n_5]}\cdot 2^{[n_2]}$, or $S_L=5^{[n_5]}\cdot1^{[n_1]}$, where ${n_2}\leq 3$ and $n_1\leq2$. Thus,
\[|S|=\sigma(S_L)\leq 5{n_5}+6\leq 5(t/5-1)+6<t+8,\]
which is a contradiction. Thus, $|S|\not=t+8$.

\bs\n {\bf Case 4:} {\em $|S|\not=t+d$ for $d\in[9,11]$.}

Assume that $|S|=t+9$. 
If ${n_3}\geq1$, then $3$ is a subsequence of $S_L$ which 
implies that $S$ contain a \zs $T$ of length $3$.
Thus $S'=S\cdot T^{-1}$ is a \zs of length $|S|-3=t+6$ 
which does not contain a \zsb of length $6$ and, equivalently,
length $t$. This contradicts Case~1, where we showed that no such 
\zs exists. Thus, ${n_3}=0$.  Similarly, $n_2=0$ (by Case~2) and $n_1=0$ (by Case~3). Consequently, $S_L=5^{[n_5]}\cdot 4^{[n_4]}$ with ${n_4}\leq 2$. 
Thus,
\[|S|=\sigma(S_L)=5{n_5}+4n_4\leq 5(t/5-1)+4(2)<t+9,\]
which is a contradiction. Thus, $|S|\not=t+9$.

Since $n_5\geq2$, $S$ contains a \zs of length $\sigma(5^{[2]})=10$. Thus, $|S|\not=t+10$.
 
Since $n_5\geq1$, $S$ contains a \zs $T$ of length $5$.
Thus $S'=S\cdot T^{-1}$ is a \zs of length $|S|-5=t+6$ 
which does not contain a \zsb of length $6$ and, equivalently, 
length $t$. This contradicts Case~1. Thus, $|S|\not=t+11$.

\bs In conclusion, we have shown that if $S$ is an arbitrary \zs over $I_3=[-3,3]$ which does not contain a \zs of length $t$, then
$|S|=t+d$ for $d\in[6,11]$. Thus, $\m_t(I_3)=t+6$.
\end{proof}
\begin{remark}
Aaron Berger~\cite{AB} has recently announced a proof of Conjecture~\ref{conj:1}.
\end{remark}
\section{Appendix}\label{sec:app}
In this section, we include Zhong's proofs of Lemma~\ref{lem:QZ} 
and Remark~\ref{rmk:QZ}.

\begin{proof}[Proof of Lemma~\ref{lem:QZ}]\
	
\n $(i)$ Since $\ls(G)\leq \m(G)$, it suffices to prove that 
$\m(G)\geq\ls(G)$. 
Let $S=\prod_{i=1}^{\ls(G)-1}g_i$ be a sequence in $\F(G)$ of 
length $|S|=\ls(G)-1$ such that $S$ has no zero-sum subsequence of 
length exp$(G)$. Assume that $\sigma(S)=h\in G$ and let $t\in\N$ be such that $(\ls(G)-1)t\equiv 1\pmod{{\rm exp}(G)}$. Then $(\ls(G)-1)th=h$ in $G$. Define $S'=\prod_{i=1}^{\ls(G)-1}(g_i-th)$. Since $\sigma(S')=\sigma(S)-(\ls(G)-1)th=0$ and $S'$ does not 
contain a zero-sum subsequence of length exp$(G)$, it follows that $\m(G)\geq \ls(G)$.

\bs\n $(ii)$ Let $S\in \B(G)$ be such that $|S|=\ls(G)-1$. We want to prove that $S$ contains a zero-sum subsequence of length $n={\rm exp}(G)$. If we assume to the contrary that $S$ does not contain a zero-sum subsequence of length $n$, then Property D implies that there exists $T\in\F(G)$ such that $S=T^{[n-1]}$. Thus, $|T|=c$
and $\sigma(T)=0$. Therefore $T^{[n/c]}$ is a zero-sum sequence 
of length $n$, a contradiction.
\end{proof}
\begin{proof}[Proof of Remark~\ref{rmk:QZ}]\
	
\n $(i)$ Let $n$ be odd and $G\cong\Z^2_n$. Since $\ls(G)=4n-3$,
then $\gcd(\ls(G)-1,n)=1$. Thus, $\ls(G)=\m(G)$ by Lemma~\ref{lem:QZ}$(i)$.

\n $(ii)$ Let $h\geq2$ be an integer and $G\cong\Z^2_{2^h}$.
Then ${\rm exp}(G)=2^h$, $\ls(G)=4(2^h-1)+1$, $\gcd(\ls(G)-1,{\rm exp}(G))=4$, and $G$ has Property $D$ (by~\cite[Theorem~3.2]{GGS}). Thus, Lemma~\ref{lem:QZ}$(ii)$ yields $\m(G)<\ls(G)$.  Since  $\gcd(\ls(G)-2,{\rm exp}(G))=1$, the proof of Lemma~\ref{lem:QZ}$(i)$ yields $\m(G)>\ls(G)-2$. Thus,  $\m(G)=\ls(G)-1$.
\end{proof}

\bs\n{\bf Acknowledgement:}\
We thank Alfred Geroldinger for providing 
references and for his valuable comments which helped 
clarify the definitions and terminology. We also thank Qinghai Zhong for allowing us to include Lemma~\ref{lem:QZ} and Remark~\ref{rmk:QZ}.

\end{document}